\documentclass{amsart}
\usepackage{all}
\pdfoutput=1

\title{Tropically unirational varieties}
\author[J.~Draisma]{Jan Draisma}
\address[Jan Draisma]{
Department of Mathematics and Computer Science\\
Technische Universiteit Eindhoven\\
P.O. Box 513, 5600 MB Eindhoven, The Netherlands\\
and Centrum voor Wiskunde en Informatica, Amsterdam,
The Netherlands}
\thanks{The first author is supported by a Vidi grant from
the Netherlands Organisation for Scientific Research (NWO)}
\email{j.draisma@tue.nl}
\author[B.J.~Frenk]{Bart Frenk}
\address[Bart Frenk]{
Department of Mathematics and Computer Science\\
Technische Universiteit Eindhoven}
\thanks{The second author is supported by a Free Competition grant from
the Netherlands Organisation for Scientific Research (NWO)}
\email{b.j.frenk@tue.nl}

\begin{document}

\begin{abstract}
We introduce {\em tropically unirational varieties}, which are
subvarieties of tori that admit dominant rational maps whose
tropicalisation is surjective. The central (and unresolved) question is
whether {\em all} unirational varieties are tropically unirational.
We present several techniques for proving tropical unirationality,
along with various examples.
\end{abstract}

\maketitle

\section{Tropical Unirationality}
Tropical geometry has proved useful for {\em implicitisation}, i.e.,
for determining equations for the image of a given polynomial or
rational map \cite{sturmfels_tevelev:2008,sturmfels_tevelev_yu:2007,sturmfels_yu:2008}. The fundamental underlying observation is that
tropicalising the map in a naive manner gives a piecewise linear map
whose image is contained in the tropical variety of the image of the
original map. Typically, this containment is strict, and for polynomial
maps with {\em generic} coefficients the difference between the two
sets was determined in \cite{sturmfels_yu:2008}.  Polynomial or rational maps arising
from applications are typically highly non-generic, and yet it would
be great if those maps could be tropicalised to determine the tropical
variety of their image. Rather than realising that ambitious goal, this
paper identifies a concrete research problem and presents several useful
tools for attacking it.

Thus let $K$ be an algebraically closed field with a non-Archimedean
valuation $v:K \to \RR \cup \{\infty\}$. We explicitly allow $v$ to be
trivial. Write $T=K^*$ for the one-dimensional torus over $K$ and $T^n$
for the $n$-dimensional torus. For a non-zero polynomial $f=\sum_\alpha
c_\alpha x^\alpha \in K[x_1,\ldots,x_m]$ we write $\Trop(f)$ for
the {\em tropicalisation} of $f$, i.e., for the function $\RR^m \to
\RR$ defined by 
\[ \Trop(f)(\xi):=\min_{\alpha} (v(c_\alpha) + x \cdot
\alpha),\ \xi \in \RR^m; \] 
here $\cdot$ stands for the standard dot product on $\RR^m$.  Throughout
this paper we will use greek letters to stand for tropical variables.

Strictly speaking, one should distinguish between a tropical
polynomial and the function that it defines, but in this paper we will
only need the latter. By Gauss's Lemma, we have $\Trop(fg)=\Trop(f)
+ \Trop(g)$ for non-zero polynomials $f,g$, and this implies that
we can extend the operator $\Trop$ to {\em rational} functions by
setting $\Trop(f/h)=\Trop(f)-\Trop(h)$. We further extend this definition to
rational maps $\varphi=(f_1,\ldots,f_n):T^m \dto T^n$ by setting $\Trop
\varphi:=(\Trop(f_1),\ldots,\Trop(f_n)):\RR^m \to \RR^n$.

If $X$ is a subvariety of $T^n$, then we write $\Trop(X)$ for the
tropicalisation of $X$, i.e., for the intersection of the corner loci
of all $\Trop(f)$ as $f$ ranges through the ideal of $X$ in $K[x_1^{\pm
1},\ldots,x_n^{\pm 1}]$. 

\begin{de}
A subvariety $X$ of $T^n$ is called {\em tropically unirational} if there
exists a natural number $p$ and a dominant rational map $\varphi:T^p
\dto X$ such that the image $\im \Trop(\varphi)$ equals $\Trop(X)$. The
map $\varphi$ is then called {\em tropically surjective}.
\end{de}

We recall that the inclusion $\im \Trop(\varphi) \subseteq \Trop(X)$
always holds \cite{draisma:2008}. The following example shows that this inclusion is
typically strict, but that $\varphi$ can sometimes be modified (at the
expense of increasing $p$) so as to make the inclusion into an equality.

\begin{ex} \label{ex:Line}
\begin{figure}
\includegraphics[width=.4\textwidth]{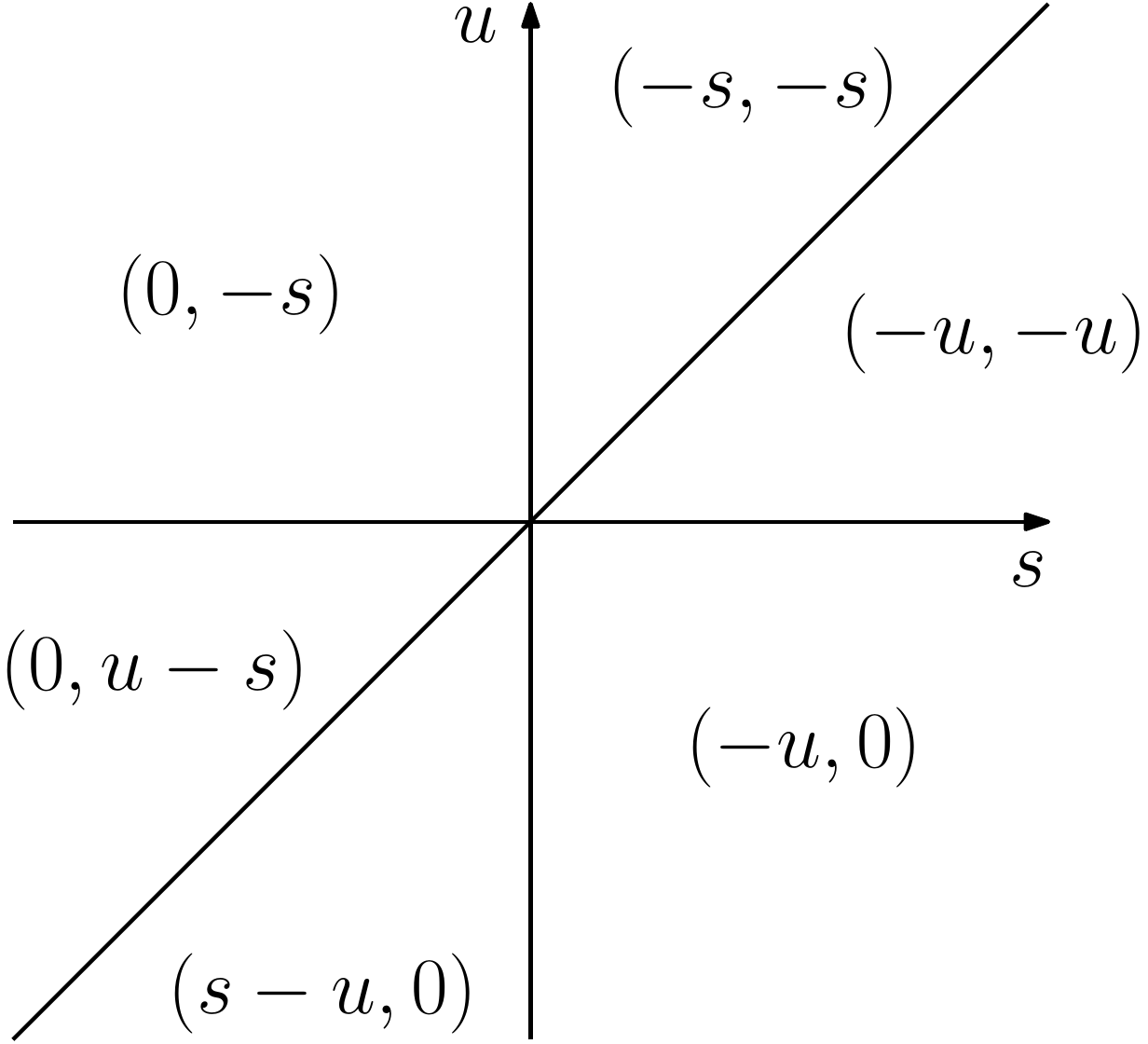}
\caption{Outside the lines, $\Trop(\psi)$ is linear and of
the indicated form.}
\label{fig:tropline}
\end{figure}
Let $X \subseteq T^2$ be the line defined by $y=x+1$, with the well-known
tripod as its tropical variety. Then the rational map $\varphi:T^1 \dto
T^2,\ t \mapsto (t,t+1)$ is dominant, but the image of its tropicalisation
only contains two of the rays of the tripod, so $\varphi$ is 
not tropically surjective. However, the map $\psi: T^2 \dto X \subseteq T^2$,
$(s,u) \mapsto (\frac{1+s}{u-s},\frac{1+u}{u-s})$ {\em is} tropically
surjective. Indeed, see Figure~\ref{fig:tropline}: under
$\Trop(\psi)$, the north-west and
south-east quadrants cover the arms of the tripod in the north and east
directions, respectively, and any of the two halfs of the north-east
quadrant covers the arm of the tripod in the south-west direction.
So $X$ is tropically unirational.  There is no tropically surjective
rational map into $X$ with $p=1$.
\end{ex}

The central question that interests us is the following.

\begin{que} \label{que:Main}
Is every unirational variety tropically unirational?
\end{que}

This paper is organised as follows. In Section~\ref{sec:Linear} we review
the known fact that (affine-)linear spaces and rational curves are
tropically unirational. In Section~\ref{sec:Combining} we prove that,
at least for {\em rational} varieties, our central question above is
equivalent to the apparently weaker question of whether $\Trop(X)$ is the
union of {\em finitely many} images $\im \Trop(\varphi_i),\ i=1,\ldots,N$
with each $\varphi_i$ a rational map $T^{p_i} \dto X$. This involves the
concept of {\em reparameterisations}: precompositions $\varphi \circ
\alpha$ of a dominant rational map $\varphi$ into $X$ with other rational
maps $\alpha$; since tropicalisation does not commute with composition,
$\Trop(\varphi \circ \alpha)$ may hit points of $\Trop(X)$ that are not
hit by $\Trop(\varphi)$.
In Section~\ref{sec:Birational} we introduce a somewhat ad-hoc technique
for finding suitable (re)parameterisations. Together with tools from
preceding sections this technique allows us, for example, to prove
that the hypersurface of singular $n \times n$-matrices is tropically
unirational for every $n$.  In Section~\ref{sec:Local} we prove that for
$X$ unirational, every sufficiently generic point on $\Trop(X)$ has a
$\dim(X)$-dimensional neighbourhood that is covered by $\Trop(\varphi)$
for suitable $\varphi$; here we require that $K$ has characteristic
zero. Combining reparameterisations, we find that there exist dominant
maps into $X$ whose tropicalisation hit full-dimensional subsets of all
full-dimensional polyhedra in the polyhedral complex $\Trop(X)$.  But more
sophisticated methods, possibly from {\em geometric tropicalisation}, will
probably be required to give a definitive answer to our central question.

\section*{Acknowledgements}
This paper was influenced by discussions with many people; among them,
Wouter Castryck, Filip Cools, Anders Jensen, Frank Sottile, Bernd
Sturmfels, and Josephine Yu. We thank all of them for their input!
The final version of this paper was written while the second author
was visiting Eva Maria Feichtner's group at the University of Bremen;
we thank them for their hospitality.

\section{Linear spaces, toric varieties, homogenisation, curves}
\label{sec:Linear}

We start with some elementary constructions of tropically unirational
subvarieties of tori.

\begin{lm} \label{lm:ToricHom}
If $X$ is a tropically unirational subvariety of $T^n$, then so is its
image $L_u \overline{\pi(X)}$ under any torus homomorphism $\pi:T^n \to
T^q$ followed by left multiplication $L_u$ with $u \in T^q$.
\end{lm}

\begin{proof}
If $\varphi:T^m \dto X$ is tropically surjective, then we claim that so is
$L_u \circ \pi \circ \varphi:T^m \dto Y:=\overline{\pi(X)}$. Indeed, since
$\phi$ is a monomial map and $L_u$ is just componentwise multiplication
with non-zero scalars, we have $\Trop(L_u \circ \pi \circ \varphi)=
\Trop(L_u) \circ \Trop(\pi) \circ \Trop(\varphi)$. Here the first map
is a translation over the componentwise valuation $v(u)$ of $u$, and
the second map is an ordinary linear map.  The claim follows from the
known fact that $\Trop(L_u) \Trop(\pi) \Trop(X)=\Trop(Y)$, which is a consequence of the main theorem of tropical geometry \cite{maclagan_sturmfels}
\end{proof}

The following is a consequence of a theorem by Yu and Yuster \cite{yu_yuster:2007}.

\begin{prop} \label{prop:Linear}
If $X$ is the intersection with $T^n$ of a linear subspace of $K^n$,
then $X$ is tropically unirational.
\end{prop}

\begin{proof}
Let $V$ be the closure of $X$ in $K^n$, by assumption a linear
subspace. The support of an element $v \in V$ is the set of $i$ such that
$v_i \neq 0$. Choose non-zero vectors $v_1,\ldots,v_p \in V$ such that
the support of each vector in $V$ contains the support of some $v_i$. Let
$A \in K^{n \times p}$ be the matrix with columns $v_1,\ldots,v_p$,
and let $v(A) \in (\RR \cup \{\infty\})^{n \times p}$ be the image of
$A$ under coordinate-wise valuation. Then Yu-Yuster's theorem states
that $\Trop(V) \subseteq (\RR \cup \{\infty\})^n$ is equal to the image
of $(\RR \cup \{\infty\})^p$ under tropical matrix multiplication with
$v(A)$. This implies that the rational map $T^p \dto T^n,\ v \mapsto Av$
is tropically surjective.
\end{proof}

Another argument for the tropical unirationality of linear spaces
will be given in Section~\ref{sec:Birational}.

\begin{lm} \label{lm:Cone}
Let $X \subseteq T^n$ be a closed subvariety, and write $\tilde{X}$ for
the the cone $\{(t ,tp) \mid t \in K^*\}$ over $X$ in $T^{n+1}$. Then
$\tilde{X}$ is tropically unirational if and only if $X$ is.
\end{lm}

\begin{proof}
If $X$ is tropically unirational, then so is $T \times X \subseteq
T^{n+1}$, and hence by Lemma~\ref{lm:ToricHom} so is the image $\tilde{X}$
of the latter variety under the torus homomorphism $(t,p) \mapsto (t,tp)$.
Conversely, if $\tilde{X}$ is tropically unirational, then so is its
image $X$ under the torus homomorphism $(t,p) \mapsto t^{-1}p$.
\end{proof}

Note that $\Trop(\tilde{X})=\{0\} \times \Trop(X) + \RR(1,\ldots,1)$;
we will use this in Section~\ref{sec:Combining}.
We can now list a few classes of tropically unirational varieties.

\begin{cor} \label{cor:Affine}
Intersections with $T^n$ of affine subspaces of $K^n$ are
tropically unirational.
\end{cor}

\begin{proof}
If $X$ the intersection with $T^n$ of an affine subspace of $K^n$, then
the cone $\tilde{X}$ is the intersection with $T^{n+1}$ of a linear
subspace of $K^{n+1}$. Thus the corollary follows from
Proposition~\ref{prop:Linear} and Lemma~\ref{lm:Cone}.
\end{proof}

The following corollary has been known at least since
Speyer's thesis \cite{speyer_thesis:2005}.

\begin{cor} \label{cor:Curves}
Rational curves are unirational.
\end{cor}

\begin{proof}
Let $\varphi=(f_1,\ldots,f_n):T \dto T^n$ be a rational map, and let $Y$
be the rational curve parameterised by it. Let $S \subseteq K$ be a finite
set containing all roots and poles of the $f_i$, so that we can write
\[ f_i(x) = c_i \prod_{s \in S}(x-s)^{e_{is}}. \]
where the $c_i$ are non-zero elements of $K$ and the $e_{is}$ are integer
exponents. Let $X \subseteq T^S$ be the image of the affine-linear
linear map $T \dto T^S$ given by $x \mapsto (x-s)_{s \in S}$. Then $X$
is tropically unirational by Corollary~\ref{cor:Affine}. Let $\pi:T^S \to
T^n$ be the torus homomorphism mapping $(z_s)_{s \in S}$ to $(\prod_{s
\in S} z_s^{e_{is}})_i$, and let $u=(c_i)_i \in T^n$. Then
the curve $Y$ is the image of $X$ under $L_u \circ \pi$, and
the corollary follows from Lemma~\ref{lm:ToricHom}.
\end{proof}

\begin{cor} \label{cor:Rank2}
The variety in $T^{m \times n}$ of $m \times n$-matrices of rank at most $2$ is
tropically unirational.
\end{cor}

\begin{proof}
Let $\varphi:T^m \times T^m \times T^n \times T^n \dto T^{m \times n}$
be the rational map defined by
\[ \varphi: (u,x,v,y) \mapsto \diag(u) (x \one^t + \one y^t) \diag(v), \]
where $\diag(u),\diag(v)$ are diagonal matrices with the entries
of $u,v$ along the diagonals; $\one^t, \one$ are the $1
\times n$ and the $m \times 1$ row vectors with all ones;
and $x,y$ are interpreted as column vectors. Elementary
linear algebra shows that $\varphi$ is dominant into the variety $Y$
of rank-at-most-$2$ matrices. Moreover, $\varphi$ is the composition of
the linear map $(u,x,y,v) \mapsto (u,x \one^t + \one y^t,v)$ with the
torus homomorphism $(u,z,v) \mapsto (\diag(u) z \diag(v))$. Hence $Y$
is tropically unirational by
Proposition~\ref{prop:Linear} and Lemma~\ref{lm:ToricHom}.
\end{proof}

\begin{cor} \label{cor:Grassmannian}
The affine cone over the Grassmannian of two-dimensional vector
subspaces of an $n$-dimensional space (or more precisely its part in
$T^{\binom{n}{2}}$ with non-zero Pl\"ucker coordinates) is tropically
unirational.
\end{cor}

\begin{proof}
The proof is identical to the proof of
Corollary~\ref{cor:Rank2}, using the rational map
\[ T^n \times T^n \mapsto T^{\binom{n}{2}},\ (u,x) 
\mapsto (u_i u_j (x_i-x_j))_{i<j}. \]
\end{proof}

Interestingly, Grassmannians of two-spaces and varieties of rank-two
matrices are among the few infinite families of varieties for which {\em
tropical bases} are known \cite{develin_santos_sturmfels:2005}. It would be nice to have a direct
link between this fact and the fact, used in the preceding proofs, that
they are obtained by smearing around a linear space with a torus action.

\begin{cor}
The varieties defined by $A$-discriminants are tropically
unirational.
\end{cor}

\begin{proof}
Like in the previous two cases, these varieties are obtained from a linear
variety by smearing around with a torus action; this is the celebrated
{\em Horn Uniformisation} \cite{Gelfand94, Kapranov91}.
\end{proof}

In fact, this linear-by-toric description of $A$-discriminants was
used in \cite{Dickenstein07} to give an efficient way to compute the
Newton polytopes of these discriminants in the hypersurface case. A
relatively expensive step in this computation is the computation of
the tropicalisation of the kernel of $A$; the state of the art for this
computation is \cite{Rincon11}.

%homogeneous and non-homogeneous varieties and maps. For
%this, let $\varphi=(f_1/D,\ldots,f_n/D): T^m \dto X \subseteq T^n$ be
%a dominant rational map into $X$, where $D$ and the $f_i$ are
%polynomials. Let $d \geq \max\{\deg f_1,\ldots,\deg f_n,\deg
%D\}$ and for $g=f_1,\ldots,f_n,g$ write $\tilde{g}:=x_{m+1}^d
%g(\frac{x_1}{x_{m+1}},\ldots,\frac{x_m}{x_{m+1}})$, the degree-$d$
%homogenisation of $g$.

%Let $\tilde{phi}:T^{m+1} \dto T^{n+1}$ be the rational map with components
%$(\tilde{f}_1,\ldots,\tilde{f}_n,\tilde{D})$ and write $\tilde{X}$
%for the Zariski closure of $\im \tilde{\varphi}$ in $T^{n+1}$.

%Then $\im \Trop(\varphi)$ is the projection into $\RR^n$ of the subset of
%$\im \Trop(\tilde{\varphi})$

\section{Combining reparameterisations}
\label{sec:Combining}

A fundamental method for constructing tropically surjective maps into a
unirational variety $X \subseteq T^n$ is precomposing one dominant map
into $X$ with suitable rational maps.

\begin{de}
Given a rational dominant map $\varphi: T^m \dto X \subseteq T^n$, a {\em
reparameterisation} of $\varphi$ is a rational map of the form $\varphi \circ
\alpha$ where  $\alpha:T^p \dto T^m$ is a dominant rational
map and $p$ is some natural number.
\end{de}

The point is that, in general, $\Trop(\varphi \circ \alpha)$ is {\em not}
equal to $\Trop(\varphi) \circ \Trop(\alpha)$. So the former tropical map has
a chance of being surjective onto $\Trop(X)$ even if the latter is not.

\begin{ex}
In Example~\ref{ex:Line}, the map $\psi:(s,t) \mapsto
(\frac{1+s}{u-s},\frac{1+u}{u-s})$ is obtained from $\varphi:t \mapsto
(t,t+1)$ by precomposing with the rational map $\alpha$ sending $(s,u)$
to $\frac{1+s}{u-s}$. Hence $\psi$ is a tropically surjective
reparameterisation of the non-tropically surjective rational map $\varphi$.
\end{ex}

This leads to the following sharpening of our Question~\ref{que:Main}.

\begin{que}
Does every dominant rational map $\varphi$ into a unirational variety $X \subseteq
T^n$ have a tropically surjective reparameterisation?
\end{que}

Note that if $X$ is rational and $\varphi:T^m \dto X \subseteq T^n$ is
birational, then {\em every} dominant rational map $\psi:T^p \to X$
factors into the dominant rational map $(\varphi^{-1} \circ \psi):T^p \dto
T^m$ and the map $\varphi$. So for such pairs $(X,\varphi)$, the preceding
question is {\em equivalent} to the question whether $X$ is tropically
unirational.

We will now show how to combine reparameterisations at the expense of
enlarging the parameterising space $T^p$. For this we need a variant
of Lemma~\ref{lm:Cone}.
Let $\varphi=(\frac{f_1}{g},\ldots,\frac{f_n}{g}):T^m \dto X
\subseteq T^n$ be a dominant rational map where $g,f_1,\ldots,f_n \in
K[x_1,\ldots,x_m]$.  Let $d>0$ be a natural number greater than or equal
to $\max\{\deg(g),\deg(f_1),\ldots,\deg(f_n)\}$, and define
the homogenisations
\begin{align*} 
\tilde{g}:=x_0^{d}
g(\frac{x_1}{x_0},\ldots,\frac{x_n}{x_0})\ \text{ and } \ 
\tilde{f}_i:=x_0^{d}
f_i(\frac{x_1}{x_0},\ldots,\frac{x_n}{x_0}),\ i=1,\ldots,n.
\end{align*}
These are homogeneous polynomials of positive degree $d$ in $n+1$ variables
$x_0,\ldots,x_n$. The map $\tilde{\varphi}: T^{m+1} \dto T^{n+1}$ with components
$(\tilde{g},\tilde{f}_1,\ldots,\tilde{f}_n)$ is called a {\em degree-$d$
homogenisation of $\varphi$}. The components of one degree-$d$
homogenisation of $\varphi$ differ from those of another by a common
factor, which is a rational function with numerator and denominator
homogeneous of the same degree. Any degree-$d$
homomgenisation of $\phi$ is dominant into the cone 
$\tilde{X}$ in $T^{n+1}$ over $X$. Recall that 
$\Trop(\tilde{X})=\{0\} \times \Trop(X) + \RR(1,\ldots,1)$. The following
lemma is the analogue of this statement for $\im \Trop(\tilde{\varphi})$.

\begin{lm} \label{lm:Homogenisation}
Let $\tilde{\varphi}$ be any degree-$d$ homogenisation of $\varphi$. Then the
image of $\Trop(\tilde{\varphi})$ equals
$\{0\} \times (\im \Trop(\varphi)) + \RR(1,\ldots,1)$. 
\end{lm}

\begin{proof}
For the inclusion $\supseteq$, let $\xi \in \RR^m$ and let $\gamma \in \RR$.
Setting $\tilde{\xi}:=(0,x)+\frac{\gamma}{d}(1,\ldots,1) \in \RR^{m+1}$ and
using that $\tilde{g}$ is homogeneous of degree $d$ we find that
\[
\Trop(\tilde{g})(\tilde{\xi})=\Trop(\tilde{g})(0,\xi)+\gamma=\Trop(g)(\xi)+\gamma; \]
and similarly for the $\tilde{f}_i$. This proves that 
\[
\Trop(\tilde{\varphi})(\tilde{\xi})=\Trop(\varphi)(\xi)+\gamma(1,\ldots,1),
\]
from which the inclusion $\supseteq$ follows. For the inclusion
$\subseteq$ let
$\tilde{\xi}=(\tilde{\xi}_0,\ldots,\tilde{\xi}_m) \in \RR^{m+1}$
and set $\xi_i:=\tilde{\xi}_i-\tilde{\xi}_0,\ i=1,\ldots,m$. Again
by homogeneity we find
\[ \Trop(\tilde{g})(\tilde{\xi})=\Trop(\tilde{g})(0,\xi)+d
\tilde{\xi}_0=\Trop(g)(\xi)+d \tilde{\xi_0} \]
and similarly for the $\tilde{f}_i$. This implies
\[ 
\Trop(\tilde{\varphi})(\tilde{\xi})=\Trop(\varphi)(\xi)+d
\tilde{\xi}_0
(1,\ldots,1), \]
which concludes the proof of $\subseteq$. 
\end{proof}

\begin{lm}[Combination Lemma] \label{lm:Combination}
Let $\varphi:T^m \dto X \subseteq T^n$ and $\alpha_i:T^{p_i}
\dto T^m$ for $i=1,2$ be dominant rational maps. Then there exists a dominant rational
map $\alpha:T^{p_1+p_2+1} \dto T^m$ such that $\im \Trop(\varphi \circ \alpha)$
contains both $\im \Trop(\varphi \circ \alpha_1)$ and $\im \Trop(\varphi
\circ \alpha_2)$.
\end{lm}

\begin{proof}
Consider a degree-$d$ homogenisation $\tilde{\varphi}:T^{m+1}
\dto \tilde{X} \subseteq T^{n+1}$ and degree-$e$ homogenisations
$\tilde{\alpha}_1:T^{p_1+1} \dto T^{m+1}, \tilde{\alpha}_2:T^{p_2+1}
\dto T^{m+1}$ of $\alpha_1,\alpha_2$, respectively. 
Define $\tilde{\alpha}:T^{p_1+1} \times T^{p_2+1} \dto T^{m+1}$ by 
\[
\alpha(\tilde{u},\tilde{v})=\tilde{\alpha}_1(\tilde{u})+\tilde{\alpha}_2(\tilde{v}).
\]
We claim that 
\[ \im \Trop(\tilde{\varphi} \circ \tilde{\alpha}) \supseteq 
\im \Trop(\tilde{\varphi} \circ \tilde{\alpha}_i) \text{ for
}i=1,2.\]
Indeed, since $\tilde{\varphi}$ has polynomial components and since
$\tilde{\alpha}_2$ is homogeneous of positive degree $e$, we have
\[ \tilde{\varphi}(\tilde{\alpha}_1(\tilde{u})+\alpha_2(\tilde{v}))
=\tilde{\varphi}(\tilde{\alpha}_1(\tilde{u}))+\text{ terms divisible
by at least one variable $\tilde{v}_j$}. \]
As a consequence, for $(\tilde{\mu},\tilde{\nu}) \in
\RR^{p_1+1} \times \RR^{p_2+1}$ we have
\[ 
\Trop(\tilde{\varphi} \circ
\tilde{\alpha})(\tilde{\mu},\tilde{\nu})
= \min \{ \Trop(\tilde{\varphi} \circ
\tilde{\alpha}_1)(\tilde{\mu}), 
\text{ terms containing at least one $\tilde{\nu}_j$}
\}.
\]
Hence if $\tilde{\mu}$ is fixed first and $\tilde{\nu}$ is
then chosen to have 
sufficiently large (positive) entries, then we find
\[ \Trop(\tilde{\varphi} \circ
\tilde{\alpha})(\tilde{\mu},\tilde{\nu})
=\Trop(\tilde{\varphi} \circ \tilde{\alpha}_1)(\tilde{\nu}). \]
This proves that $\im \Trop(\tilde{\phi} \circ
\tilde{\alpha}_1) \subseteq \im \Trop(\tilde{\phi} \circ
\tilde{\alpha})$. Repeating the argument with the roles of
$1$ and $2$  reversed proves the claim. 

Now we carefully de-homogenise as follows. 
First, a straightforward computation shows that $\tilde{\varphi} \circ
\tilde{\alpha}_i$ is a degree-$de$ homogenisation of $\varphi \circ
\alpha_i$ for $i=1,2$, hence by Lemma~\ref{lm:Homogenisation} we have
\[ \im (\Trop(\tilde{\varphi} \circ \tilde{\alpha}_i))
=\{0\} \times \im \Trop(\varphi \circ \alpha_i) + \RR(1,\ldots,1). \]
Similarly, writing $\tilde{\alpha}=(a_0,\ldots,a_m):T^{p_1+1+p_2+1}
\dto T^m$ for the components of $\tilde{\alpha}$ we define
$\alpha:T^{p_1+p_2+1} \dto T^m$ as the de-homogenisation of
$\tilde{\alpha}$ given by
\[ \alpha(u,\tilde{v})
= \left(\frac{a_1(1,u,\tilde{v})}{a_0(1,u,\tilde{v})},
\ldots,
\frac{a_m(1,u,\tilde{v})}{a_0(1,u,\tilde{v})} \right). \]
A straightforward computation shows that $\tilde{\phi} \circ 
\tilde{\alpha}$ is a degree-$de$ homogenisation of $\phi
\circ \alpha$. Hence by Lemma~\ref{lm:Homogenisation} we
have 
\[ \im (\Trop(\tilde{\varphi} \circ \tilde{\alpha}))
=\{0\} \times \im \Trop(\varphi \circ \alpha) + \RR(1,\ldots,1). \]
Now the desired containment 
\[ \im \Trop(\phi \circ \alpha) \supseteq \im \Trop(\phi \circ
\alpha_i) \text{ for } i=1,2 \]
follows from
\begin{align*} 
\{0\} \times \im \Trop(\phi \circ \alpha) 
&= (\{0\} \times \RR^n) \cap \im \Trop(\tilde{\phi} \circ
\tilde{\alpha}) \\
&\supseteq
(\{0\} \times \RR^n) \cap \im \Trop(\tilde{\phi} \circ
\tilde{\alpha}_i) \\
&=
\{0\} \times \im \Trop(\phi \circ \alpha_i) 
\end{align*}
\end{proof}

%\section{Reparameterising maps from the plane to itself} \

%\note{Look at the Mathematica notebook}

\section{Birational projections} \label{sec:Birational}

%\note{Look at Mathematica notebook for some examples. Is there something more general to be said about determinantal varieties?}

In this section we show that rational subvarieties of $T^n$ that have
sufficiently many birational toric projections are tropically unirational.
Here is a first observation.

\begin{lem}
Let $X \subseteq T^n$ be an algebraic variety and $\pi: T^n \rightarrow
T^d$ a torus homomorphism whose restriction to $X$ is birational, with
rational inverse $\varphi$. Then $\Trop(\pi) \circ \Trop(\varphi)$
is the identify on $\RR^d$.
\end{lem}

\begin{proof}
Let $\eta \in \RR^d$ be a point where $\Trop(\varphi)$ is
(affine-)linear. Such
points form the complement of a codimension-$1$ subset and are therefore
dense in $\RR^d$. Hence it suffices to prove that $\Trop(\varphi)(\eta)$
maps to $\eta$ under $\Trop(\pi)$.  Let $y \in T^d$ be a point with
$v(y)=\eta$ where $\varphi$ is defined and such that $x:=\varphi(y) \in X$
satisfies
$\pi(x)=y$. Such points exist because
$v^{-1}(\eta)$ is Zariski-dense in $T^d$. Now
$\xi:=v(x)$ equals $\Trop(\varphi)(\eta)$ by
linearity of $\Trop(\varphi)$ at $\eta$ and 
$\Trop(\pi)\xi=\eta$ by linearity of $\Trop(\pi)$. 
\end{proof}

For our criterion we need the following terminology.

\begin{de}
Let $P \subseteq \RR^n$ be a $d$-dimensional polyhedron and let $A:\RR^n
\to \RR^d$ be a linear map. Then $P$ is called {\em $A$-horizontal}
if $\dim AP = d$.
\end{de}

\begin{prop} \label{prop:BirationalProjection}
Let $X \subseteq T^n$ be an algebraic variety and $\pi: T^n \rightarrow
T^d$ a torus homomorphism whose restriction to $X$ is birational,
with rational inverse $\varphi$. Using the Bieri-Groves theorem,
write $\Trop(X)=\bigcup_i P_i$ where the $P_i$ are finitely many
$d$-dimensional polyhedra. Then $\im \Trop(\varphi)$ is the union of
all $\Trop(\pi)$-horizontal polyhedra $P_i$.
\end{prop}

\begin{proof}
Let $P_i$ be a $\Trop(\pi)$-horizontal polyhedron. We want to prove that
$\Trop(\varphi) \circ \Trop(\pi)$ is the identity on $P_i$.  To this
end, let $\xi \in P_i$ be such that $\Trop(\varphi)$ is
affine-linear at
$\eta:=\Trop(\pi)\xi$. The fact that $P_i$ is $\Trop(\pi)$-horizontal
implies that such $\xi$ are dense in $P_i$. To prove that
$\Trop(\varphi)(\eta)$ equals $\xi$ let $x \in X$ be a
point with $v(x)=\xi$ such that $\varphi$ is defined at $y:=\pi(x)$
and satisfies $\varphi(y)=x$. The existence of such a point follows
from birationality and the density of fibers in $X$ of the valuation
map \cite{payne:2009}. Now $\eta$ equals $v(y)$ by linearity of $\Trop(\pi)$
and $\xi=v(x)=v(\varphi(y))$ equals $\Trop(\varphi)(\eta)$ by linearity
of $\Trop(\varphi)$ at $\eta$.  Hence $\Trop(\varphi) \circ \Trop(\pi)$
is the identity on $P_i$, as claimed. Thus $\im \Trop(\varphi)$ contains
$P_i$. Since the projections of $\Trop(\pi)$-horizontal polyhedra $P_i$
together form all of $\RR^d$, we also find that $\im \Trop(\varphi)$
does not contain any points outside those polyhedra.
\end{proof}

%\begin{lem}
%Then $\Trop(\varphi)$ is a section of
%$\Trop(\pi)$, i.e., $\Trop(\pi) \circ \Trop(\varphi)$ equals the identity
%on $\RR^d$. Furthermore, the image of $\Trop(\varphi)$ contains the set
%of points $\{ \xi \in \Trop X \mid \mbox{$\{\xi\}$ is a fiber of $\Trop
%\pi$} \}$ as a dense subset.
%\end{lem}
%
%\begin{proof}
%Denote the coordinate permutation $T^d \times T^n \rightarrow T^n \times
%T^d$ that maps $(u, v)$ to $(v, u)$ by $\sigma$. By birationality the
%image of the graph $\Gamma_{\varphi}$ under $\sigma$ is a dense subset
%of $\Gamma_{\pi}$ and hence both have equal tropicalisation $T$. Let
%$\xi \in \RR^d$. We prove that its image under $\Trop \pi \circ \Trop
%\varphi$ is $\xi$. Clearly,
%\[
%\big(\Trop(\varphi)(\xi), \Trop(\pi)(\Trop(\varphi)(\xi)\big) \in T.
%\]
%and, since $\pi$ is a torus homomorphism, $T = \Gamma_{\Trop(\pi)}$. The first statement now follows from $(\Trop(\varphi)(\xi), \xi) \in T$.
%
%%For the second statement, let
%\[
%I = \mathcal{I}(\Gamma_{\pi}) \subseteq K[x_1^{\pm}, \ldots, x_d^{\pm}, y_1^{\pm}, \ldots, y_n^{\pm}].
%\]
%For each $y_i$ there are $f_i, g_i \in K[x_1^{\pm}, \ldots, x_d^{\pm}]$
%such that $y_i f_i - g_i \in I$. Thus, outside a polyhedral subset
%of $\RR^d$ of positive codimension the valuation of $y_i$ is uniquely
%determined by the valuations of $x_1, \ldots, x_d$. The statement follows
%since the complement of a finite union of polyhedral subsets of positive
%codimension is dense.
%\end{proof}

\begin{cor}\label{cor:dense_unirational}
Let $X \subseteq T^n$ be a birational
variety and write $\Trop(X)=\bigcup_i P_i$ as in
Proposition~\ref{prop:BirationalProjection}. If for each $P_i$ there
exists a torus homomorphism $\pi: T^n \to T^d$ that is birational on $X$
and for which $P_i$ is $\Trop(\pi)$-horizontal, then $X$ is
tropically unirational.
\end{cor}

\begin{proof}
In that case, there exist finitely many homomorphisms
$\pi_1,\ldots,\pi_N:T^n \to T^d$, birational when restricted to $X$,
such that each $P_i$ is $\Trop(\pi_j)$-horizontal for at least one $j$.
Then Proposition~\ref{prop:BirationalProjection} shows that
the rational inverse $\varphi_j$ of $\pi_j$ satisfies $P_i
\subseteq \im \Trop(\varphi_j)$. Now use the Combination
Lemma~\ref{lm:Combination}.
\end{proof}

In particular, when all coordinate projections to tori of dimension
$\dim X$ are birational the variety $X$ is tropically unirational. This
is the case in the following statement.

\begin{cor}
For any natural number $n$ the variety of singular $n \times n$-matrices
is tropically unirational.
\end{cor}

\begin{proof}
A matrix entry $m_{ij}$ of a singular matrix can be expressed as a
rational function of all other $n^2-1$ entries (with denominator equal
to the corresponding $(n-1) \times (n-1)$-subdeterminant). This shows
that the map $T^{n^2} \to T^{n^2-1}$ forgetting $m_{ij}$ is birational.
Any $(m^2-1)$-dimensional polyhedron in $\RR^{m^2}$ is horizontal with
respect to some coordinate projection $\RR^{n^2}
\to \RR^{n^2-1}$, and this holds {\em a fortiori} for the cones of
tropically singular matrices. Now apply
Corollary~\ref{cor:dense_unirational}.
\end{proof}

Corollary \ref{cor:dense_unirational} also gives an
alternative proof of Corollary~\ref{cor:Affine} stating that
tropicalisations of affine-linear spaces are tropically
unirational. 

\begin{proof}[Second proof of Corollary~\ref{cor:Affine}]
Let $X$ be the intersection with $T^n$ of a $d$-dimensional
linear space in $K^n$. For each polyhedron $P_i$ of $\Trop(X)$
there exists a coordinate projection $\pi:T^n \to T^I$, with $I$
some cardinality-$d$ subset of the coordinates, such that $P_i$ is
$\Trop(\pi_I)$-horizontal. Here we have not yet used that
$X$ is affine-linear. Then the restriction $\pi_I:X \to T^I$
is dominant, and since $X$ is affine-linear, it is also birational.
Now apply Corollary~\ref{cor:dense_unirational}.
\end{proof}

We continue with an example of a determinantal variety of codimension
larger than one whose unirationality is a consequence of Corollary
\ref{cor:dense_unirational}.

\begin{ex}
Let $V \subseteq M_{4 \times 5}(K)$ be the variety of matrices of rank smaller than or equal to $3$. The ideal of $V$ contains all maximal minors and the dimension of $V$ equals $18$. One way to see the latter statement is to write a matrix in $V$ in the following form,
\[
\left(\begin{array}{cc}
A & B\\C & D \end{array}\right), \qquad A \in M_{3 \times 3},\quad B \in M_{3 \times 2},\quad C \in M_{1 \times 3},\quad D \in M_{1 \times 2}.
\]
There are no conditions on $A$, $B$ and $C$, while $D$ is uniquely determined by the choice of $A, B$ and $C$. The dimension thus equals $3 \cdot 3 + 3 \cdot 2 + 1 \cdot 3 = 18$.

Let $(m_{ij})$ denote the standard coordinate functions on $M_{4 \times 5}$. We aim to show that the projection into any subset of $X$ of size $18$ is birational. Let $z_1 = m_{i,j}$ and $z_2 = m_{l,k}$ be the indices of the coordinate functions left out of the projection. Note that if $z_1$ appears in a maximal minor, in which $z_2$ doesn't, then $z_1$ is a rational function of the coordinate functions in the maximal minor. In particular, if $z_1$ and $z_2$ are in different columns, there exist such maximal minors for $z_1$ and $z_2$ and hence both are rational in the remaining $18$ coordinate function.

The case that $z_1$ and $z_2$ are in the same column requires some calculation. Suppose without loss of generality that $z_1 = m_{3,4}$ and $z_2 = m_{4,5}$. Then,
\begin{eqnarray*}
0 & = & m_{35} \det M_{124,234} + m_{45} \det M_{123,234} + m_{15} \det M_{234,234} + m_{25} \det M_{124,234},\\
0 & = & m_{35} \det M_{124,134} + m_{45} \det M_{123,134} + m_{15} \det M_{234,134} + m_{25} \det M_{134,134}
\end{eqnarray*}
by cofactor expansion of the determinants of the matrices $M_{1234,2345}$ and $M_{1234,1345}$. The set of equations has a unique solution for $m_{35}$ and $m_{45}$ when
\[
\det M_{124,134} \det M_{123,134} \neq \det M_{124,234} \det M_{123,134},
\]
showing that the projection is birational.
\end{ex}

We conclude this section with a beautiful example, suggested to us by
Filip Cools and Bernd Sturmfels, and worked out in collaboration with
Wouter Castryck and Filip Cools.

\begin{ex}
Let $Y \subseteq T^5$ be parameterised by $(s^4,s^3 t,\ldots,t^4),\
(s,t) \in T^2$, the affine cone over the rational normal quartic. Write
$X:=\overline{Y+Y} \subseteq T^5$, the first secant variety.  Writing
$z_0,\ldots,z_4$ for the coordinates on $T^5$, $X$ is the hyperplane
defined by
\begin{align*} 
&\det
\begin{bmatrix}
z_0 & z_1 & z_2\\
z_1 & z_2 & z_3\\
z_2 & z_3 & z_4
\end{bmatrix}\\
&=z_0 z_2 z_4 + 2 z_1z_2z_3-z_1^2 z_4-z_0 z_3^2-z_2^3\\
&=a + 2b - c - d - e.
\end{align*}
This polynomial is homogeneous both with respect to the ordinary grading
of $K[z_0,\ldots,z_4]$ and with respect to the grading where $z_i$
gets degree $i$. Hence its Newton polygon is three-dimensional; see
Figure~\ref{fig:HankelNewton}. Modulo its two-dimensional lineality space
the tropical variety $\Trop(X)$ is two-dimensional. Intersecting with a
sphere yields Figure~\ref{fig:HankelTropVar}. 

\begin{figure}
\includegraphics[width=.4\textwidth]{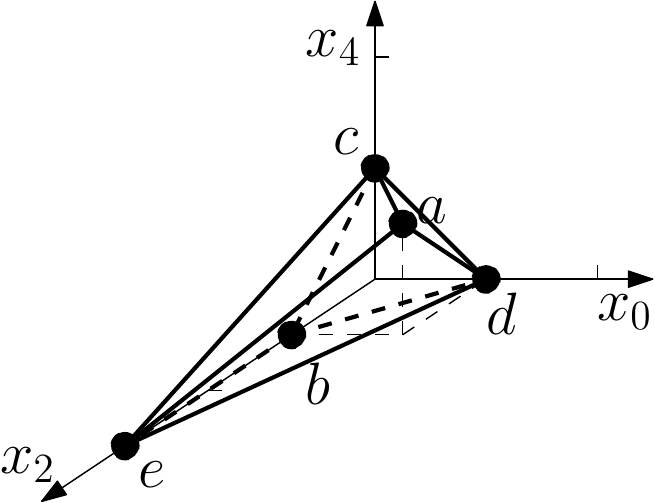}
\caption{The Newton polytope of the Hankel determinant.
Only the exponents of $z_0,z_2,z_4$ have been drawn.}
\label{fig:HankelNewton}
\end{figure}
\begin{figure}
\includegraphics[width=.6\textwidth]{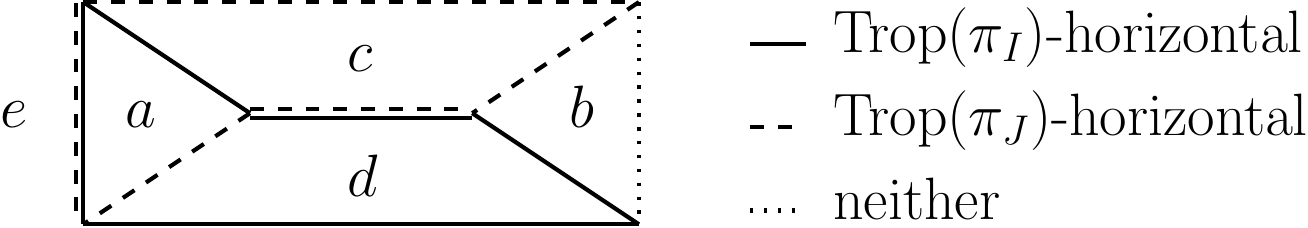}
\caption{The tropical variety of the Hankel determinant.  Two-dimensional
regions correspond to the monomials $a,b,c,d,e$.}
\label{fig:HankelTropVar}
\end{figure}

Now set $I:=\{1,2,3,4\}$
and $J:=\{0,1,2,3\}$. Then the coordinate projections $\pi_I:T^5 \to
T^I,\ \pi_J:=T^5 \to T^I$ are birational, since $z_0$ and $z_4$ appear
only linearly in the Hankel determinant. Let $P$ be a full-dimensional
cone in $\Trop(X)$, and let $\{\alpha,\beta\}$ be the corresponding
edge in the Newton polygon. Then $P$ is $\Trop(\pi_I)$-horizontal if
and only if $\alpha_0 \neq \beta_0$ and $\Trop(\pi_J)$-horizontal if
and only if $\alpha_4 \neq \beta_4$. Figure~\ref{fig:HankelTropVar}
shows that all but one of the cones are, indeed, horizontal with respect
to one of these projections. Let $P$ denote the cone corresponding
to the edge between the monomials $b=z_1z_2z_3$ and $e=z_2^3$.
By Proposition~\ref{prop:BirationalProjection} and the Combination
Lemma~\ref{lm:Combination}, there exists a rational parameterisation
of $X$ whose tropicalisation covers all cones of $\Trop(X)$ except,
possibly, $P$. We now set out to find a parameterisation
whose tropicalisation covers $P$.

Let $\zeta \in P$. By Lemma \ref{lm:Homogenisation} we may assume that $\zeta$ is of the form $(\zeta_0, \zeta_1, 2 \zeta_1, 3 \zeta_1, \zeta_2)$. We aim to show that there exist two reparametrisations of $\phi: T^4 \rightarrow X$, where
\[
\phi(u_0,u_1,v_0,v_1) = (u_0 + v_1, u_0 v_0  + u_1 v_1, u_0 v_0^2 + u_1 v_1^2, u_0 v_0^3 + u_1 v_1^3, u_0 v_0^4 + u_1 v_1^4).
\]
such that $\zeta$ is in the image of at least one of them. Note that from the defining inequalities of $P$ it follows that $\zeta_0 \geq 0$ and $\zeta_2 \geq 4 \zeta_1$. Let $i \in K$ be a fourth root of unity and consider the map $\psi: T^3 \rightarrow T^4$ defined by
\[
\psi(x_0, x_1, x_2) = \big(1 + x_0, -1, i x_1, -x_1(1+x_2 x_1^{-4})\big).
\]
A short computation shows that the restriction of the tropicalisation of $\phi \circ \psi$ to the cone defined by $\xi_0 \geq 0$, $\xi_2 \geq 4 \xi_1$ and $\xi_2 \leq \xi_0 + 4 \xi_1$ is the linear function $(\xi_0, \xi_1, \xi_2) \mapsto (\xi_0, \xi_1, 2 \xi_2, 3 \xi_1, \xi_2)$. If $\zeta$ satisfies $\zeta_2 \leq \zeta_0 + 4 \zeta_1$ then the image of $(\zeta_0, \zeta_1, \zeta_2)$ under this tropicalisation is exactly $\zeta$.

If $\zeta$ satisfies $\zeta_2 \geq \zeta_0 + 4 \zeta_1$ it is in the image of the tropicalisation of $\phi \circ \psi \circ \iota$, where
\[
\iota(x_0, x_1, x_2) = (x_0^{-1}, x_1^{-1}, x_2^{-1}).
\]
The tropicalisation is linear on the cone $0 \geq \xi_0$, $4 \xi_1 \geq \xi_2$ and $\xi_2 \geq \xi_0 + \xi_1$ and maps $-(\zeta_0, \zeta_1, \zeta_2)$ to $\zeta$.
\end{ex}

\section{Very local reparameterisations} \label{sec:Local}

Let $X \subseteq \T^n$ be a $d$-dimensional rational variety that is the
closure of the image of a rational map $\varphi:T^m \dto T^n$. Suppose
without loss of generality that $X$ is defined over a valued field
$(K, v)$ such that $v(K^*)$ is a finite dimensional vector space over
$\QQ$. Write $\bar{\xi}$ for the image of $\xi \in \RR$ under the
canonical projection $\RR \rightarrow \RR/v(K^*)$. We can now state the
main result of this section and its corollary.

\begin{thm}\label{thm:local_param}
Assume that the field $K$ has characteristic zero.
Let $\xi \in \Trop(X)$ and set $d$ to be the dimension of the
$\QQ$-vectorspace spanned by $\bar{\xi}_1, \ldots, \bar{\xi}_n$. There
exists a rational map $\alpha:T^d \dto T^m$ and an open subset $U
\subseteq \RR^d$ such that the restriction of $\Trop(
\varphi \circ \alpha)$ to $U$ is an injective affine linear map, whose image contains
$\xi$.
\end{thm}

\begin{cor}
Assume that $K$ has characteristic zero.
Let $\{C_1, \ldots, C_k\}$ be a finite set of $v(K^*)$-rational polyhedra of dimension $\dim X$ such that
\[
\Trop(X) = \bigcup_{i=1}^k C_i.
\]
There exist a natural number $p$ and a rational map $\alpha:T^p \dto T^m$
such that the image of $\Trop(\varphi \circ \alpha)$ intersects each $C_i$
in a $\dim X$-dimensional subset.
\end{cor}
\begin{proof}
By the theorem, for each cone $C_i$ there exists a reparametrisation $\alpha_i$ such that the tropicalization of $\varphi \circ \alpha_i$ hits $C_i$ in a full dimensional subset. They can be combined using the Combination Lemma.
\end{proof}

The main step in the proof of the theorem is Proposition \ref{prop:denseness}, which is a valuation theoretic result.

Let $\xi$ be a point of $\RR^n$. Such a point defines a valuation $v_{\xi}$ on the field of rational functions $L = K(y_1, \ldots, y_n)$ of $T^n$ by
\[
v_{\xi}(h) = \Trop(h)(\xi), \qquad h \in L
\]
Let $L_{\xi}$ denote the completion of $L$ with respect to $v_{\xi}$
and denote its algebraic closure by $\overline{L_{\xi}}$. That closure is
equipped with the unique valuation whose restriction to $L_\xi$ equals
$v_{\xi}$ \cite[\S 144]{Waerden67}. Denote by $K[y_1^{\QQ}, \ldots,
y_n^{\QQ}]$ the subring of $\overline{L_{\xi}}$ generated by all roots of
the elements $y_1,\ldots,y_n$.

The next lemmata deal with the case $n = 1$. They allow us to prove
Proposition~\ref{prop:denseness} below by means of induction on the
number of variables.

\begin{lem}\label{lm:dense1}
Let $\xi \in \RR$ such that $\xi$ is not in the $\QQ$-vectorspace spanned by $v(K^*)$. Then $K[t^{-1},t]$ is dense in $K(t)_{\xi}$ 
\end{lem}
\begin{proof}
Let $p/q \in K(t)$. If $q$ is a monomial we are done. Suppose it isn't. The valuation on $q$ is of the form $v_{\xi}(q) = \min_i v(p_i) + i\xi$. Moreover, the minimum is attained exactly once since otherwise $\xi$ would be a $\QQ$-multiple of some element of $v(K^*)$. Say it is attained at $j$. Compute,
\begin{eqnarray*}
\frac{p}{q}
& = & \frac{p}{a_j t^j + (q - a_j t^j)}\\[0.5em]
& = & a_j t^j \frac{p}{1 + (q - a_j t^j)/(a_j t^j)}\\[0.5em]
& = & a_j t^j p \sum_{n = 0}^{\infty} \Big(\frac{(q - a_j t^j)}{a_j t^j}\Big)^n
\end{eqnarray*}
The convergence of the power series with respect to $w$ is a consequence of $w(q - a_j t^j) > w(a_j t^j)$. The limit is easily seen to to coincide with $p/q$. This completes the proof.
\end{proof}

\begin{lem}\label{lm:dense2}
Suppose $K$ is algebraically closed of characteristic $0$.
Then $K[t^{\QQ}]$ is dense in $\overline{K(t)_{\xi}}$, 
\end{lem}
\begin{proof}
Denote the residue field of $K(t)_{\xi}$ by $k$. Note that it is also the residue field of $K(t)$ under $\xi$, by the conditions on $\xi$.

We prove by induction on $d$ that all zeroes in $\overline{K(t)_{\xi}}$ of a polynomial of degree $d$ over $K(t)_{\xi}$ can be approximated arbitrarily well with elements of $K[t^\QQ]$. For $d=1$ this is the content of Lemma \ref{lm:dense1}. Assume that the statement is true for all degrees lower than $d$. We follow the proof of \cite[\S 14, Satz]{Waerden73}.

Let $P(S)=S^d + a_{d-1} S^{d-1} + \ldots + a_0 \in K(t)_{\xi}[S]$. After a coordinate change replacing $S$ by $S-\frac{1}{d} a_{d-1}$ we may
assume that $a_{d-1}=0$. Indeed, a root $s$ of the original polynomial can be approximated well by elements of $K[t^\QQ]$ if and only if the
root $s+\frac{1}{d} a_{d-1}$ can be approximated well, since $a_{d-1}$ itself can be approximated well.

If now all $a_i$ are zero, then we are done. Otherwise, let the minimum among the numbers $v(a_{d-i})/i$ be $\omega+q\tau$, where $\omega \in
v(K^*)$ and $q \in \QQ$, and let $c$ be a constant in $K$ with valuation $\omega$. Setting $S=ct^q U$ transforms $P$ into
\[ c^d t^{dq} (U^d + b_{d-2} U^{d-2} + \ldots + b_0), \] where each $b_i$ is an element of $K(t^{1/p})_{\xi}$ of valuation
at least zero, with $p$ the denominator of $q$. Moreover, some $b_i$ has valuation zero. Let $Q(U)$ denote the polynomial in the brackets.
The image of $Q(U)$ in the polynomial ring $k[U]$ over the residue field is neither $U^d$ as $b_i$ has non-zero image in $L$, nor a $d$-th power
of an other linear form as the coefficient of $U^{d-1}$ is zero. Hence the image of $Q(U)$ in $k[U]$ has at least two distinct roots in the
algebraically closed residue field $k$, and therefore
factors over $k$ into two relatively prime polynomials. By
Hensel's lemma \cite[\S 144]{Waerden73}, $Q$ itself factors over $K(t^{1/p})_{\xi}$ into two polynomials of positive degree. By induction the roots of these polynomials can be approximated arbitrarily well by elements of $K[t^\QQ]$, hence so can the roots of $Q$ and of $P$.
\end{proof}

\begin{prop}\label{prop:denseness}
Let $(K, v)$ be an algebraically closed field of
characteristic $0$ with valuation $v$ and $\xi \in \RR^n$
whose entries are $\QQ$-linearly independent over $\RR /
v(K^*)$. Then $K[y_1^{\QQ}, \ldots, y_n^{\QQ}]$ is dense in
$\overline{L_{\xi}}$.
\end{prop}
\begin{proof}
Follows from Lemma \ref{lm:dense2} by induction on the number of variables.
\end{proof}

We are now ready to prove the main result.

\begin{proof}[Proof of Theorem \ref{thm:local_param}]
Choose $\tau_1, \ldots, \tau_d \in \RR$ such that their projections in $\RR/v(K^*)$ form a basis of the $\QQ$-vectorspace spanned by $\bar{\xi_1}, \ldots, \bar{\xi_n}$. Let $t_1, \ldots, t_n$ be variables and denote by $L$ the field $K(t_1, \ldots, t_n)$ equipped with the unique valuation $w$ that extends $v$ and satisfies $w(t_i) = \tau_i$.

There exists a point $x' \in T^m(\overline{L_{\xi}})$ such that $w(\varphi(x')) = \xi$. By Proposition \ref{prop:denseness} there exists an approximation $x \in T^m(K[t_1^{\QQ}, \ldots, t_d^{\QQ}]$) of $x'$ that satisfies $w(\varphi(x)) = \xi$. Choose $e \in \NN$ such that every coefficient of $x$ is already in $K[t_1^{\pm 1/e}, \ldots, t_d^{\pm 1/e}]$ and set $s_i = t_i^{1/e}$. Thus, $x \in T^m(K[s_1^{\pm}, \ldots, s_d^{\pm}])$, and hence defines a rational map $T^d \rightarrow T^m$. Denote this map $\alpha$. We show that there exists a neighbourhood of $\sigma = \frac{1}{e} (\tau_1, \ldots, \tau_d)$ such that the restriction of $\Trop(\varphi \circ \alpha)$ to $U$ satisfies the conclusions of the theorem.

First, note that, by construction of $\alpha$,  $\Trop(\varphi \circ \alpha)(\sigma) = \xi$.  Every component $\varphi_i(x)=\varphi_i(\alpha(s_1,\ldots,s_d))$ of $\varphi$ is a Laurent polynomial over $K$ in the $s_j$ with a unique term $z_i s_1^{b_{i,1}} \cdots s_d^{b_{i,d}}$ of minimal valuation $\xi_i=v(z_i)+\sum_{j=1}^d b_{i,j} \sigma_j$. If we let $\sigma$ vary in a small neighbourhood and change the valuations of the $s_i$ accordingly, then for each $i$ the same term of $\varphi_i$ has the minimal valuation. Hence $\TT(\varphi \circ \alpha)$ is linear at $\sigma$ with differential the matrix $(b_{ij})$. Finally, as the numbers $\eta_1,\ldots,\eta_n$ span the same $\QQ$-space as $\sigma_1,\ldots,\sigma_d$ modulo $v(K^*)$, the matrix $(b_{ij})$ has full rank $d$. This completes the proof.
\end{proof}

\section{Concluding remarks} \label{sec:Conclusion}

The concept of tropical surjectivity of a rational map seems natural and concrete, and, as far as we know, not to have been studied before. This paper presented some methods of determining whether a rational map is tropically surjective, and aims to be a starting point for further study. In particular, the question whether every unirational variety is tropically unirational is still open. It seems likely that tehniques from geometric tropicalisation will prove useful in making further progress on this question.

\bibliographystyle{plain}
\bibliography{tropunirational}

\end{document}